\pdfoutput=1
\documentclass[10pt,twoside]{article}

\usepackage{lmodern}
\usepackage[T1]{fontenc}
\usepackage[utf8]{inputenc}
\usepackage{textcomp}
\usepackage[english]{babel}

\usepackage{amsmath}
\usepackage{amssymb}
\usepackage{amsthm}
\usepackage{mathtools}
\usepackage{bigstrut}
\usepackage{cases}

\newcommand{\fulltitle}{Generalized linearization techniques in electrical impedance tomography}
\newcommand{\shorttitle}{generalized linearization in eit}

\usepackage[top=31mm, bottom=31mm, left=33mm, right=33mm]{geometry}

\usepackage{fancyhdr}
\newcommand{\N}{\mathbf{N}}
\newcommand{\R}{\mathbf{R}}
\newcommand{\C}{\mathbf{C}}
\newcommand{\E}{\mathbf{E}}

\DeclareMathOperator*{\argmin}{arg\,min}

\DeclareMathOperator*{\inv}{inv}

\newtheorem{theorem}{Theorem}
\newtheorem{proposition}{Proposition}
\newtheorem{corollary}{Corollary}
\newtheorem{lemma}{Lemma}
\newtheorem{definition}{Definition}
\newtheorem{example}{Example}
\newtheorem{remark}{Remark}

\mathtoolsset{showonlyrefs,showmanualtags}

\usepackage[pdftex,colorlinks,pdftitle={\fulltitle}]{hyperref}
\hypersetup{linkcolor=black}
\hypersetup{urlcolor=black}
\hypersetup{citecolor=black}

\begin{document}

\title{\fulltitle}
\author{Nuutti Hyv\"{o}nen\footnote{Aalto University, Department of Mathematics and Systems Analysis, P.O.\ Box 11100, FI-00076 Aalto, Finland (nuutti.hyvonen@aalto.fi).} \and Lauri Mustonen\footnote{Aalto University, Department of Mathematics and Systems Analysis, P.O.\ Box 11100, FI-00076 Aalto, Finland (lauri.mustonen@aalto.fi).}}
\date{May 30, 2017}

\maketitle

\begin{abstract}
Electrical impedance tomography aims at reconstructing the interior electrical conductivity from surface measurements of currents and voltages.
As the current-voltage pairs depend nonlinearly on the conductivity, impedance tomography leads to a nonlinear inverse problem.
Often, the forward problem is linearized with respect to the conductivity and the resulting linear inverse problem is regarded as a subproblem in an iterative algorithm or as a simple reconstruction method as such.
In this paper, we compare this basic linearization approach to linearizations with respect to the resistivity or the logarithm of the conductivity.
It is numerically demonstrated that the conductivity linearization often results in compromised accuracy in both forward and inverse computations.
Inspired by these observations, we present and analyze a new linearization technique which is based on the logarithm of the Neumann-to-Dirichlet operator.
The method is directly applicable to discrete settings, including the complete electrode model.
We also consider Fr\'echet derivatives of the logarithmic operators.
Numerical examples indicate that the proposed method is an accurate way of linearizing the problem of electrical impedance tomography.

\bigskip


\end{abstract}

\section{Introduction}

The reconstruction task of electrical impedance tomography (EIT) is undoubtedly one of the most studied nonlinear inverse boundary value problems \cite{Uhlmann09}.
The mathematical foundations of the problem were introduced by Calder{\'o}n in \cite{Calderon80}, which considers the conductivity reconstruction based on the Dirichlet-to-Neumann operator.
As has become a more standard numerical approach, in this paper we consider the Neumann-to-Dirichlet operator, which maps the boundary currents to the corresponding boundary potentials for a given interior conductivity.
The dependence of the Neumann-to-Dirichlet operator on the conductivity is nonlinear so that reconstructing the conductivity from the measurements is a nonlinear, and also illposed, inverse problem.

Nonlinear inverse problems can be straightforwardly approached with nonlinear least squares minimization algorithms, which rely on successive linearizations of the forward operator. In particular, the accuracy of the sequential linearizations certainly affects the performance of such an iterative reconstruction method.
In EIT, one-step linearization can also be used to obtain an approximative reconstruction \cite{Adler09,Cheney90,Harrach12}.
The Fr\'echet derivatives of the EIT forward operator with respect to the conductivity can be computed explicitly \cite{Cheney90,Kaipio00,Lechleiter08b}.
By applying the chain rule of Banach spaces, or via direct computation, it is possible to consider derivatives with respect to the resistivity or the logarithm of the conductivity as well.
In this paper, we compare the linearization errors resulting from these different input parametrizations of the electrical properties.
We also discuss the corresponding methods for the complete electrode model (CEM), which is an accurate model for practical EIT measurements \cite{Cheng89,Somersalo92}.

As the main novelty, we introduce the logarithm of the Neumann-to-Dirichlet operator and use its differentiability properties to construct a new linearization method for EIT.
Loosely speaking, the traditional forward operator is replaced by an operator that maps the logarithm of the conductivity to the logarithm of the boundary potentials, the latter being understood in a sense of linear operators.
The corresponding least squares inversion algorithm then fits the computed logarithmic potentials to the logarithm of the measurement matrix.
Although this logarithmic forward operator is still nonlinear, numerical experiments show that the resulting linearization errors are in most cases smaller than with any other considered linearization method.

This paper is organized as follows.
In Section~\ref{sec:param}, the continuum forward model of EIT and the related Neumann-to-Dirichlet operator are reviewed.
We write the dependence of the Neumann-to-Dirichlet operator on the electrical properties in three different ways and recall also the corresponding Fr\'echet derivatives.
Section~\ref{sec:log} presents the formal definition for the logarithm of the Neumann-to-Dirichlet operator and studies its differentiability and other properties as an unbounded operator on the space of square-integrable functions.
In Section~\ref{sec:cem}, observations from Sections \ref{sec:param} and \ref{sec:log} are generalized for the complete electrode model.
Numerical experiments concerning linearization errors for both forward and inverse computations are given in Section~\ref{sec:numerics}.
Finally, Section~\ref{sec:discussion} offers concluding remarks.

\section{Input parametrization in EIT}
\label{sec:param}

The continuum forward model of EIT for the electric potential $u$ is written as
  \begin{align}
  \label{eq:eitcont}
  \nabla \cdot (\sigma \nabla u) &= 0 \qquad \text{in } \varOmega, \\
  \sigma \frac{\partial u}{\partial \nu} &= f \qquad \text{on } \partial \varOmega,
  \end{align}
where $\varOmega \subset \R^d$, $d \geq 2$, is a bounded domain with a Lipschitz boundary $\partial \varOmega$ and a connected complement, the electrical conductivity $\sigma \in L_+^\infty(\varOmega)$ is real-valued and isotropic, and $f \in H_\diamond^{-1/2}(\partial \varOmega)$ can be complex-valued.
The conductivity is bounded from below by a positive constant, that is,
  \begin{equation}
  \label{eq:linfplus}
  L_+^\infty(\varOmega) \coloneqq \left\{ v \in L^\infty(\varOmega; \R) : v \geq a \text{ a.e.~in } \varOmega \text{ for some } a>0 \right\},
  \end{equation}
and the boundary current density belongs to the mean-free Sobolev space $H_\diamond^{-1/2}(\partial \varOmega)$ defined via
  \begin{equation}
  \label{eq:l2diam}
  H_\diamond^{r}(\partial \varOmega) \coloneqq \left\{ v \in H^{r}(\partial \varOmega) : \langle 1, v\rangle = 0 \right\}, \qquad r \in (-1,1),
  \end{equation}
due to the conservation of electric charge. Here and in what follows, the bracket $\langle \cdot, \cdot \rangle \colon H^{-r}(\partial \varOmega)\times H^{r}(\partial \varOmega) \to \C$ denotes the {\em sesquilinear} dual pairing that has an interpretation as an extension of the  $L^2(\partial \varOmega)$ inner product $(\cdot, \cdot) \colon L^2(\partial \varOmega) \times L^2(\partial \varOmega) \to \C$. In particular, apart from $L^\infty(\varOmega) := L^\infty(\varOmega; \R)$, the multiplier field for all employed function spaces is $\C$.

The variational form of the Neumann problem \eqref{eq:eitcont} is to find $u \in H^1(\varOmega)$ such that
  \begin{equation}
  \label{eq:varcont}
  \int_\varOmega \sigma \nabla u \cdot \nabla \overline{v} \, {\rm d} x = \big\langle f, v|_{\partial \varOmega} \big\rangle
  \end{equation}
holds for all $v \in H^1(\varOmega)$.
The standard theory for elliptic partial differential equations states that there exists a unique solution for \eqref{eq:varcont} in the quotient space $H^1(\varOmega) / \mathbf{C}$  for any given current density $f \in H^{-1/2}_\diamond(\varOmega)$. In particular, there is a unique mean-free boundary potential
  \begin{equation}
  \label{eq:ubnd}
  U \coloneqq u|_{\partial \varOmega} \in H_\diamond^{1/2}(\partial \varOmega)
  \end{equation}
that depends linearly and boundedly on the corresponding $f \in H^{-1/2}_\diamond(\varOmega)$.~\cite{Grisvard85}

The linear map $f \mapsto U$, which obviously depends on $\sigma$, is called the Neumann-to-Dirichlet operator and is denoted by $\varLambda(\sigma)$. For any given $\sigma \in L_+^\infty(\varOmega)$, the mapping $$\varLambda(\sigma) \colon H_\diamond^{-1/2}(\partial \varOmega) \to H_\diamond^{1/2}(\partial \varOmega)$$ is a linear isomorphism that is self-adjoint in the sense that
\begin{equation}
\label{eq:sadjoint}
\langle f, \varLambda(\sigma) g \rangle = \overline{\langle g, \varLambda(\sigma) f \rangle} \qquad {\rm for} \ {\rm all} \ f, g \in H^{-1/2}_{\diamond}(\partial \varOmega) 
\end{equation}
and positive,
\begin{equation}
\label{eq:positivity}
\langle f, \varLambda(\sigma) f \rangle \geq c \| f \|_{H^{-1/2}(\partial \varOmega)}^2 \qquad {\rm for} \ {\rm all} \ f \in H^{-1/2}_{\diamond}(\partial \varOmega) \ {\rm and} \ {\rm some} \ c>0, 
\end{equation}
as can be easily deduced from \eqref{eq:varcont} and (Neumann) trace theorems for those elements of $H^1(\varOmega)/\mathbf{C}$ for which the range of $\nabla \cdot \sigma \nabla(\cdot)$ is a subspace of $L^2(\varOmega)$ (cf.,~e.g.,~\cite[p.~381, Lemma 1]{Dautray88}).
Unless stated otherwise, we interpret $\varLambda(\sigma)$ as a self-adjoint operator from $L^2_{\diamond}(\partial \varOmega)$ to itself, that is, as a map
$$
\varLambda(\sigma) \colon f \mapsto U, \quad L^2_{\diamond}(\partial \varOmega) \to L^2_{\diamond}(\partial \varOmega),
$$
which is compact as are the (dense) embeddings $H_\diamond^{1/2}(\partial \varOmega) \hookrightarrow L^2_{\diamond}(\partial \varOmega) \hookrightarrow H_\diamond^{-1/2}(\partial \varOmega)$. In particular,  $\varLambda(\sigma)$  admits a spectral decomposition
  \begin{equation}
  \label{eq:spectral}
  \varLambda(\sigma) f = \sum_{k=1}^\infty \lambda_k (f, \phi_k) \, \phi_k,
  \end{equation}
where the eigenvalues satisfy $\lambda_{k} \geq \lambda_{k+1}$ and $\R_+ \ni \lambda_k \to 0$ as $k \to \infty$, and the corresponding eigenfunctions $\{\phi_k\}_{k=1}^\infty$ form an orthonormal basis for $L_\diamond^2(\partial \varOmega)$.

Let us denote by $\mathcal{L}(L_\diamond^2(\partial \varOmega))$ the Banach space of bounded linear operators from $L_\diamond^2(\partial \varOmega)$ to itself. The mapping
\begin{equation}
\label{eq:lambda}
\varLambda \colon \sigma \mapsto \varLambda(\sigma), \quad L_+^\infty(\varOmega) \to \mathcal{L}(L_\diamond^2(\partial \varOmega))
\end{equation}
is nonlinear and is called the (continuum model) \emph{forward operator} of EIT. The inverse problem of EIT is to find $\sigma$ from the knowledge of $\varLambda(\sigma)$. In practice, this means that the determination of the conductivity is based on current-voltage pairs measured on the boundary.
The following example shows that by composing the forward operator with an elementary function, one can obtain another forward operator that in some cases depends more linearly on the electrical input parameters.
  \begin{example}[Constant conductivity]
  \label{ex:const}
  If $\sigma \equiv \inv(\rho) \coloneqq \rho^{-1}$ for $\rho \in \mathbf{R}_+$, then the potential $u$ from \eqref{eq:eitcont} solves the Laplace equation $\Delta u = 0$ with the normal derivative $\rho f$ on the boundary. In particular, the mapping $\rho \mapsto \varLambda(\sigma) = (\varLambda \circ \inv)(\rho)$ is linear when restricted to spatially constant functions.
  \end{example}
Of course, the example above is nothing but the Ohm's law, stating that in a homogeneous medium the potential depends linearly on the resistivity $\rho$.
Motivated by the example, we denote
\begin{equation}
\label{eq:star}
\varLambda_{\rm inv} \coloneqq \varLambda \circ \inv \colon \rho \mapsto \varLambda(\rho^{-1}), \quad L_+^\infty(\varOmega) \to \mathcal{L}(L_\diamond^2(\partial \varOmega)).
\end{equation}
Similarly, the logarithm of the conductivity is denoted by $\kappa \coloneqq \log(\sigma)$ and the corresponding composite forward operator is defined as
\begin{equation}
\label{eq:starstar}
\varLambda_{\exp} \coloneqq \varLambda \circ \exp \colon \kappa \mapsto \varLambda(\mathrm{e}^\kappa), \quad L^\infty(\varOmega) \to \mathcal{L}(L_\diamond^2(\partial \varOmega)).
\end{equation}
The advantage of considering the logarithm of the conductivity is that the domain of $\varLambda_{\exp}$ is the natural $L^\infty$-space without positivity constraints, simplifying many optimization schemes in numerical computations.
On the other hand, the advantage of using the standard conductivity parametrization with the operator $\varLambda$ is the simplicity of the corresponding sesquilinear forms, for example in the context of stochastic Galerkin methods \cite{Leinonen14}.
Note that considering the logarithm of the resistivity would not add anything new compared to \eqref{eq:starstar}, except for a sign change.

It is well known that the map $\sigma \mapsto \varLambda(\sigma)$ is Fr\'echet differentiable,~i.e.,~that the bounded derivative $$D\varLambda(\sigma;\eta) \colon L^2_\diamond(\partial \varOmega) \to L^2_\diamond(\partial \varOmega)$$ exists for every $\sigma \in L^\infty_+(\varOmega)$ and depends linearly and boundedly on $\eta \in L^\infty(\varOmega)$ in the topology of  $\mathcal{L}(L_\diamond^2(\partial \varOmega))$ (see,~e.g.,~\cite[Section~3]{Lechleiter08b}).
In fact, the derivative can be obtained from the equation
\begin{equation}
\label{eq:sigmaderiv}
\big( f, D \varLambda(\sigma; \eta) g  \big) = -\int_\varOmega \eta \nabla u_f \cdot \nabla \overline{u}_g \, {\rm d} x, \qquad f,g \in L^2_\diamond(\partial \varOmega),
\end{equation}
where $u_f \in H^1(\varOmega)$ denotes the solution to the forward problem \eqref{eq:eitcont} with the conductivity $\sigma$ and current density $f$.
By using the chain rule of differentiation for Banach spaces and \eqref{eq:sigmaderiv}, it easily follows that the alternative parametrizations of the forward operator \eqref{eq:star} and \eqref{eq:starstar} are also Fr\'echet differentiable and the corresponding derivatives can be characterized by 
\begin{equation}
\label{eq:rhoderiv}
\big( f, D \varLambda_{\rm inv}(\rho; \eta) g \big) = \int_\varOmega \frac{\eta}{\rho^2} \nabla u_f \cdot \nabla \overline{u}_g \, {\rm d} x
\end{equation}
and
\begin{equation}
\label{eq:kappaderiv}
\big( f, D \varLambda_{\exp}(\kappa; \eta) g \big) = -\int_\varOmega \eta \, \mathrm{e}^\kappa \nabla u_f \cdot \nabla \overline{u}_g \, {\rm d} x
\end{equation}
for $f,g \in L^2_\diamond(\partial \varOmega)$.

\section{Logarithmic forward operator}
\label{sec:log}

In this section, we first introduce the logarithm of the Neumann-to-Dirichlet map $\varLambda(\sigma)$ as an unbounded operator on $L^2_\diamond(\partial \varOmega)$ and subsequently consider its differentiability with respect to the conductivity.

\subsection{Formal definition}
\label{ssec:logdef}

The spectral representation \eqref{eq:spectral} allows a simple way of defining a logarithm for the Neumann-to-Dirichlet operator:
  \begin{equation}
  \label{eq:logspectral}
  \log\!\varLambda(\sigma) \colon f \mapsto \sum_{k=1}^\infty \log(\lambda_k) (f,\phi_k) \, \phi_k,
  \end{equation}
where $\log$ denotes the principal branch of the natural logarithm.
As the eigenvalues $\{\lambda_k\}_{k=1}^\infty \subset \R_+$ accumulate at zero, $\log \! \varLambda(\sigma)$ is not bounded as an operator from $L^2_\diamond(\partial \varOmega)$ to itself. We define the domain of $\log \! \varLambda(\sigma)$ to simply be
\begin{equation}
\label{eq:domlog}
\mathcal{D}\big(\log\!\varLambda(\sigma)\big)  = \bigg\{g \in L^2_\diamond(\partial \varOmega) \ : \ \| \log\!\varLambda(\sigma)g \|_{L^2(\partial \varOmega)}^2 =  
\sum_{k=1}^\infty  \log^2(\lambda_k) |(g,\phi_k)|^2 < \infty \bigg\}.
\end{equation}
It is obvious that $\mathcal{D}(\log \! \varLambda(\sigma))$ is a dense linear subspace of $L^2_\diamond(\partial \varOmega)$ and $\log \! \varLambda(\sigma) f \in L^2_\diamond(\partial \varOmega)$ for any $f \in \mathcal{D}(\log \! \varLambda(\sigma))$.

\begin{proposition}
The logarithmic Neumann-to-Dirichlet operator $\log\!\varLambda(\sigma)$ defined by \eqref{eq:logspectral} and \eqref{eq:domlog} is a self-adjoint (unbounded) operator on $L^2_\diamond(\partial \varOmega)$.
\end{proposition}

\begin{proof}
Since $\mathcal{D}(\log\!\varLambda(\sigma))$ is dense in $L^2_\diamond(\partial \varOmega)$, the adjoint operator $$\log\!\varLambda(\sigma)^* \colon \mathcal{D}(\log \! \varLambda(\sigma)^*) \to L^2_\diamond(\partial \varOmega)$$ is well defined~\cite[p.~196, Thm.~1]{Yosida80}. With the help of the Cauchy--Schwarz inequality, it is easy to check that
\begin{equation}
\label{eq:symmetry}
(\log\!\varLambda(\sigma) f, g) = (f, \log\!\varLambda(\sigma) g)
\end{equation}
is finite for all $f, g \in \mathcal{D}(\log\!\varLambda(\sigma))$. In  other words, $\log\!\varLambda(\sigma)$ is symmetric, that is, $\log\!\varLambda(\sigma)^* = \log\!\varLambda(\sigma)$ on $\mathcal{D}(\log\!\varLambda(\sigma)) \subset \mathcal{D}(\log\!\varLambda(\sigma)^*)$.

Let $V:=\mathcal{N}(\log\!\varLambda(\sigma))  \subset \mathcal{D}(\log\!\varLambda(\sigma))$ be the nullspace of $\log\!\varLambda(\sigma)$. Note that $V$ is either finite-dimensional or the trivial subspace, depending on whether $\lambda=1$ is an eigenvalue of $\varLambda (\sigma)$. In consequence, $L^2_\diamond(\partial \varOmega) = V \oplus W$, where $W:= V^\perp$ is the orthogonal complement of $V$. It follows from \eqref{eq:logspectral} and \eqref{eq:domlog} that the range of $\log\!\varLambda(\sigma)$ is exactly $W$, and thus $\log\!\varLambda(\sigma)$ is self-adjoint as an operator from $\mathcal{D}(\log\!\varLambda(\sigma))\cap W$ to $W$~\cite[p.~199, Corollary]{Yosida80}. 

Let $g \in \mathcal{D}(\log\!\varLambda(\sigma)^*)$ be arbitrary,~i.e.,~there exists $g^* \in L^2_\diamond(\partial \varOmega)$ such that  
$$
(\log\!\varLambda(\sigma) f, g) = (f, g^*)  \qquad {\rm for} \ {\rm all} \ f \in \mathcal{D}(\log\!\varLambda(\sigma)).
$$
Decompose $g = g_0 + g_1$, with $g_0 \in V \subset \mathcal{D}(\log\!\varLambda(\sigma))$ and $g_1 \in W$, and observe that
$$
(f, g^*) = (\log\!\varLambda(\sigma) f, g_0) +
(\log\!\varLambda(\sigma) f, g_1) = (\log\!\varLambda(\sigma) f, g_1) 
$$
for all $f \in \mathcal{D}(\log\!\varLambda(\sigma))$ due to \eqref{eq:symmetry}. Since $\log\!\varLambda(\sigma)$ is self-adjoint on $\mathcal{D}(\log\!\varLambda(\sigma))\cap W$, $g_1$ must belong to $\mathcal{D}(\log\!\varLambda(\sigma))\cap W$, and altogether we have $g =  g_0 + g_1 \in \mathcal{D}(\log\!\varLambda(\sigma)) \subset  \mathcal{D}(\log\!\varLambda(\sigma)^*)$. In other words, $\mathcal{D}(\log\!\varLambda(\sigma)^*) = \mathcal{D}(\log\!\varLambda(\sigma))$, which completes the proof.
\end{proof}

Because any self-adjoint operator is also closed~\cite{Yosida80}, $\mathcal{D}(\log\!\varLambda(\sigma))$ becomes a Hilbert space if equipped with the graph norm
\begin{equation}
\label{eq:lognorm}
\| g \|_{\mathcal{G}(\log\! \varLambda(\sigma))} = \left(\| g \|_{L^2(\partial \varOmega)}^2
+   \| \log\!\varLambda (\sigma)  g \|_{L^2(\partial \varOmega)}^2\right)^{1/2},
\end{equation}
with respect to which $\log\!\varLambda (\sigma)$ is trivially a bounded operator.

\begin{corollary}
\label{corollary}
The logarithmic Neumann-to-Dirichlet map $\log\!\varLambda(\sigma)$ defined by \eqref{eq:logspectral} and \eqref{eq:domlog} can be interpreted as  a compact operator
$$
\log \! \varLambda(\sigma) \colon H^{\epsilon}_{\diamond}(\partial \varOmega) \to L^2_{\diamond}(\partial \varOmega), \qquad \epsilon > 0,   
$$
that coincides with its $L^2_\diamond(\partial \varOmega)$-adjoint on $H^\epsilon_{\diamond}(\partial \varOmega)$. 
\end{corollary}

The proof of Corollary~\ref{corollary} is based on the following lemma. Observe that the lemma could be extended (with obvious modifications) for $-1/2 \leq s \leq 1/2$ by duality and for a wider scale of smoothness indices by utilizing some integer power of $\varLambda(\sigma)^{-1}$ in place of $\varLambda(\sigma)^{-1/2}$ in the following proof, assuming $\partial \varOmega$ is smooth enough.

\begin{lemma}
\label{lemma:sobscale}
For any fixed $\sigma \in L^\infty_+(\varOmega)$, it holds that
$$
H^{s}_{\diamond}(\partial \varOmega) = \big\{ g \in L^2_\diamond(\partial \varOmega) \ : \ \|g\|_{{s}} < \infty \big\}, \qquad 0 \leq s \leq \frac{1}{2},
$$
where
$$
 \|g\|_{s} := \bigg( \sum_{k=1}^\infty \frac{1}{\lambda_k^{2s}}  |(g,\phi_k)|^2 \bigg)^{1/2}
$$
defines an equivalent norm for $H^{s}_{\diamond}(\partial \varOmega)$.
\end{lemma}

\begin{proof}
Let us introduce the positive powers of $\varLambda(\sigma) \colon L^2_\diamond(\partial \varOmega) \to L^2_\diamond(\partial \varOmega)$ in the natural way, that is,
$$
\varLambda^s(\sigma) \colon f  \mapsto \sum_{k=1}^\infty \lambda_k^{s} (f,\phi_k) \, \phi_k, \quad  L^2_\diamond(\partial \varOmega) \to L^2_\diamond(\partial \varOmega),
$$
which defines a compact, injective, self-adjoint operator with a dense range for any $s > 0$. The negative powers are the corresponding inverse operators: 
$$
\varLambda^{-s}(\sigma) \colon f  \mapsto \sum_{k=1}^\infty \frac{1}{\lambda_k^{s}} (f,\phi_k) \, \phi_k, \quad \mathcal{D}\big(\varLambda^{-s}(\sigma)\big) \to L^2_\diamond(\partial \varOmega), \qquad s > 0, 
$$
where (cf.,~e.g.,~\cite[Theorem~2.8]{Engl96})
$$
\mathcal{D}\big(\varLambda^{-s}(\sigma)\big) := \mathcal{R}\big(\varLambda^{s}(\sigma)\big) =  \big\{g \in L^2_\diamond(\partial \varOmega) \ : \ \|\varLambda^{-s}(\sigma)g\|_{L^2(\partial \varOmega)} = \| g \|_{s} < \infty \big\}.
$$
Using the same arguments as for $\log\!\varLambda(\sigma)$ above, it is trivial to deduce that $\varLambda^{-s}(\sigma) \colon \mathcal{D}(\varLambda^{-s}(\sigma)) \to L^2_\diamond(\partial \varOmega)$ is self-adjoint and becomes an isomorphism between Hilbert spaces if its domain is equipped with the graph norm defined via
\begin{equation}
\label{eq:equivalence}
\| g \|_{s}  \leq \| g \|_{\mathcal{G}(\varLambda^{-s}(\sigma))} := \Big(\| g \|_{L^2(\partial \varOmega)}^2
+   \| \varLambda^{-s} (\sigma)  g \|_{L^2(\partial \varOmega)}^2\Big)^{1/2} \leq C
\| g \|_{s}
\end{equation}
for $g \in \mathcal{D}(\varLambda^{-s}(\sigma))$ and with $C>0$.

Since $\varLambda(\sigma) \colon H^{-1/2}_\diamond(\partial \varOmega) \to H^{1/2}_\diamond(\partial \varOmega)$ is positive and self-adjoint, the (positive) square root $\varLambda^{1/2}(\sigma)$ is, in fact, an isomorphism from $L^2_\diamond(\partial \varOmega)$ to $H^{1/2}_\diamond(\partial \varOmega)$; see,~e.g.,~\cite[Lemma~3.4]{Bruhl01} for a simple proof. In particular, $\mathcal{D}(\varLambda^{-1/2}(\sigma)) = H^{1/2}_\diamond(\partial \varOmega)$.
According to \cite[p.~10, Definition~2.1 \& Remark~2.3]{Lions72} and the definition of $H^{s}(\partial \varOmega)$ as a (complex) interpolation space (see,~e.g.,~\cite[p.~36, Theorem~7.7]{Lions72}), we thus have
$$
H^{s}_\diamond(\partial \varOmega) = \big[H^{1/2}_\diamond(\partial \varOmega), L^2_{\diamond}(\partial \varOmega) \big]_{1-2s} = \mathcal{D}\big(\varLambda^{(1-(1-2s))/2}(\sigma)\big) = \mathcal{D}\big( \varLambda^{-s}(\sigma)\big), \qquad 0 \leq s \leq 1/2,
$$
with the graph norm of $\mathcal{D}( \varLambda^{-s}(\sigma))$ being equivalent to that of $H^{s}_\diamond(\partial \varOmega)$. The claim now follows from \eqref{eq:equivalence}.
\end{proof}

\noindent
{\em Proof of Corollary \ref{corollary}}.
By virtue of Lemma~\ref{lemma:sobscale} and the asymptotic dominance of $\lambda_k^{-\epsilon}$ over $\log^2(\lambda_k)$ as $k$ tends to infinity,
$$
\| \log\!\varLambda(\sigma) f\|_{L^2(\partial \varOmega)} \leq \| f \|_{\mathcal{G}(\log\! \varLambda(\sigma))} \leq C \| f \|_{\mathcal{G}(\varLambda^{-\epsilon/2}(\sigma))} \leq C \| f \|_{\epsilon/2} \leq C \| f \|_{H^{\epsilon/2}_\diamond(\partial \varOmega)}
$$
for any $\epsilon > 0$ and $f \in H^{\epsilon/2}_\diamond(\partial \varOmega)$, with $C = C(\epsilon)>0$ that may change between different occurrences. Since $\log\!\varLambda(\sigma) \colon L^2_\diamond(\partial \varOmega) \supset \mathcal{D}(\log\!\varLambda(\sigma)) \to L^2_\diamond(\partial \varOmega)$ is self-adjoint and the embedding $H^{\epsilon}_\diamond(\partial \varOmega) \hookrightarrow H^{\epsilon/2}_\diamond(\partial \varOmega)$ is compact, the assertion follows. \hfill $\square$

So far we have discussed the logarithm of the linear Neumann-to-Dirichlet operator.
A natural way to define the corresponding nonlinear forward operator is as follows (cf.\ \eqref{eq:lambda} and \eqref{eq:starstar}):
\begin{definition}
\label{def:logarithmic}
The logarithmic forward operator of EIT is
  \begin{equation}
  L \coloneqq \log\!\varLambda \circ \exp \colon \kappa \mapsto \log\!\varLambda(\mathrm{e}^\kappa), \quad L^\infty(\varOmega) \to \mathcal{L}(H^{\epsilon}_\diamond(\partial \varOmega), L^2_\diamond(\partial \varOmega))
\end{equation}
for some $\epsilon > 0$.
\end{definition}
Considering the conductivity logarithm $\kappa$ instead of $\sigma$ affords the same convenience as with the non-logarithmic forward operator $\varLambda_{\exp}$ in \eqref{eq:starstar}.
Moreover, the linear nature of $\varLambda_{\rm inv}$ in \eqref{eq:star} with constant conductivities (see~Example~\ref{ex:const}) becomes now available as well, as shown next.

\begin{example}[Constant conductivity revisited]
\label{ex:const2}
If $\sigma \equiv \mathrm{e}^\kappa$ is constant in $\varOmega$, then the mapping $\kappa \mapsto L(\kappa)$ is affine.
\end{example}
\begin{proof}
As in Example~\ref{ex:const}, the potential $u$ from \eqref{eq:eitcont} solves the Laplace equation with the normal derivative $\mathrm{e}^{-\kappa} f$.
Thus, the Neumann-to-Dirichlet operator satisfies
\begin{equation}
\varLambda(\sigma)f = \sum_{k=1}^\infty \lambda_k \mathrm{e}^{-\kappa} (f, \phi_k) \phi_k,
\end{equation}
where the eigensystem corresponds to the unit conductivity.
Taking the logarithm results in
\begin{equation}
L(\kappa)f = \log\!\varLambda(\sigma)f = \sum_{k=1}^\infty \log(\lambda_k) (f, \phi_k) \phi_k - \kappa \sum_{k=1}^\infty (f, \phi_k) \phi_k = (L(0) - \kappa \, \mathrm{id}) f,
\end{equation}
which proves the claim.
\end{proof}

The following example shows that the logarithmic forward operator depends ``almost linearly'' on the conductivity logarithm of an interior inclusion.
\begin{example}[Nested concentric disk]
\label{ex:inclusion}
Let $\varOmega \subset \mathbf{R}^2$ be the unit disk and let $\varOmega_R \subset \varOmega$ be an origin-centered disk with radius $0 < R < 1$.
Define the conductivity via its logarithm $\kappa = \log(\sigma)$ as
  \begin{equation}
  \kappa =  \begin{cases}
           0, &\mathrm{in} \;\varOmega \setminus \varOmega_R, \\
           \widetilde{\kappa}, &\mathrm{in} \;\varOmega_R
           \end{cases}
  \end{equation}
for some $\widetilde{\kappa} \in \R$.
Now the logarithmic forward operator satisfies $L(\kappa) = L(0) + \widetilde{\kappa} \, L'(0) + O(\widetilde{\kappa}^3)$ for some $L'(0) \in \mathcal{L}(L^2_\diamond(\partial \varOmega))$ and with $O(\widetilde{\kappa}^3)$ referring to the topology of $\mathcal{L}(L^2_\diamond(\partial \varOmega))$.
\end{example}
\begin{proof}
It is shown in \cite{Isaacson1986} that in this rotationally symmetric case the eigenfunctions of $\varLambda(\sigma)$ do not depend on $\widetilde{\kappa}$, and the eigenvalues are
  \begin{equation}
  \lambda_{2k-1} = \lambda_{2k} = \frac{1}{k} \frac{g_k(\widetilde{\kappa})}{g_k(-\widetilde{\kappa})}, \qquad k=1,2,3,\ldots,
  \end{equation}
where
\begin{equation}
g_k(\widetilde{\kappa}) \coloneqq 1 - \frac{(\mathrm{e}^{\tilde{\kappa}}-1)}{\mathrm{e}^{\tilde{\kappa}}+1} R^{2k}.
\end{equation}
The logarithm of the eigenvalue $\lambda_{2k}$ satisfies
\begin{equation}
\label{eq:odd}
\log(\lambda_{2k}) + \log(k) = \log(g_k(\widetilde{\kappa})) - \log(g_k(-\widetilde{\kappa})),
\end{equation}
which is an analytic, odd function with respect to $\tilde{\kappa}$.
In particular, the second derivative of $\log(\lambda_{2k})$ with respect to $\widetilde{\kappa}$ vanishes at the origin, and it is also straightforward, yet tedious to check that the corresponding first and third derivatives are uniformly bounded over $k \in \N$ in a neighborhood of the origin. The claim thus follows by plugging \eqref{eq:odd} in \eqref{eq:logspectral} and using Taylor's theorem.
\end{proof}

\subsection{Fr\'echet differentiability}
\label{ssec:logdiff}
Let us next consider the differentiability of $\log \! \varLambda(\sigma)$ with respect to the conductivity. Since $\log \! \varLambda(\sigma)$ has not itself been introduced as a boundary operator corresponding to some (elliptic) partial differential equation, but its definition directly involves an eigensystem for $\varLambda(\sigma)$, the natural way to start would be to consider the differentiability of $\{ \lambda_k, \phi_k\}_{k=1}^\infty \subset \R_+ \times L^2_\diamond(\partial \varOmega)$. Indeed, the Gateaux differentiability of (a suitable parametrization for) individual eigenvalues and eigenfunctions with respect to the conductivity could  be proven by considering the corresponding eigensystem $\{\lambda_k^{-1}, \phi_k\}_{k=1}^\infty$ of the (unbounded) Dirichlet-to-Neumann operator $\varLambda^{-1}(\sigma)$ and utilizing its analyticity  with respect to perturbations in the conductivity; see,~e.g.,~\cite{Calderon80} for the analyticity of  $\varLambda^{-1}(\,\cdot\,)$ and \cite[p.~392, Theorem~3.9]{Kato95} or \cite{Kriegl11} for more information about the differentiability of an eigensystem for a self-adjoint, unbounded operator with respect to a real parameter that does not affect the domain of definition. However, due to the singularity of the logarithm at the origin and the infiniteness of the sum in \eqref{eq:logspectral}, (the natural sense of) the differentiability of $\log \! \varLambda(\sigma)$ is not trivial to establish. Moreover, as our main motivation for differentiating $\log \! \varLambda(\sigma)$ with respect to the conductivity is numerical computing as in Section~\ref{sec:numerics}, we only consider derivatives for finite-dimensional approximations of $\sigma \mapsto \log \! \varLambda(\sigma)$ in what follows.

To be more precise, we investigate a finite-dimensional, positive and self-adjoint mapping
\begin{equation}
\label{eq:lambdan}
\varLambda^{(n)}(\sigma) := P^{(n)}\varLambda(\sigma) P^{(n)} \colon L^2_\diamond (\partial \varOmega) \to L^2_\diamond (\partial \varOmega), \qquad n \in \N,
\end{equation}
where 
\begin{equation}
\label{eq:projspan}
P^{(n)} = \sum_{k=1}^nP_k \colon L^2_\diamond (\partial \varOmega) \to {\rm span} \{ \psi_1, \dots , \psi_n \}
\end{equation}
is an $L^2(\partial \varOmega)$-orthogonal projection composed of individual projections $P_k$ onto fixed orthonormal basis functions $\psi_k \in L^2_\diamond(\partial \varOmega)$, $k=1, \dots, n$, respectively. Furthermore, let 
$$
\varLambda^{(n)}(\sigma) \colon f  \mapsto \sum_{k=1}^n \mu_k (f, \varphi_k) \, \varphi_k
$$
be a spectral representation for $\varLambda^{(n)}(\sigma)$, with $\{ \mu_k \}_{k=1}^n \subset \R_+$ and $\{ \varphi_k \}_{k=1}^n \subset {\rm span}\{\psi_1, \dots , \psi_n \}$ being orthonormal.
Take note that $\{ \mu_k \}_{k=1}^n$ and $\{ \varphi_k \}_{k=1}^n$ depend on $\sigma$, but the basis for the discretization $\{ \psi_k \}_{k=1}^n$ does {\em not}. The logarithm of $\varLambda^{(n)}(\sigma)$ is defined in the same way as for its infinite-dimensional counterpart:
\begin{equation}
  \label{eq:logspectral2}
  \log\!\varLambda^{(n)}(\sigma) \colon f \mapsto \sum_{k=1}^n \log(\mu_k) (f, \varphi_k) \, \varphi_k ,
\end{equation}
which is obviously a bounded, linear map from $L^2_\diamond(\partial \varOmega)$ to itself.

\begin{theorem}
\label{thm:Fder}
The mapping 
\begin{equation}
\label{eq:logfmap}
L^\infty_+(\varOmega) \ni \sigma \mapsto \log \! \varLambda^{(n)}(\sigma) \in \mathcal{L}(L_\diamond^2(\partial \varOmega))
\end{equation}
is Fr\'echet differentiable. The derivative at $\sigma \in L^\infty_+(\varOmega)$ in the direction $\eta \in L^\infty(\varOmega)$ is the element of $\mathcal{L}(L_\diamond^2(\partial \varOmega))$ defined via
\begin{equation}
\label{eq:derrepre}
D \log\!\varLambda^{(n)}(\sigma; \eta): f \mapsto \sum_{j=1}^n \sum_{k=1}^n \frac{\log(\mu_j) - \log(\mu_k)}{\mu_j - \mu_k} (f, \varphi_k) \big( D\varLambda(\sigma; \eta)\varphi_k, \varphi_j \big) \, \varphi_j
\end{equation}
where $D\varLambda(\sigma;\eta): L^2_\diamond(\partial \varOmega) \to L^2_\diamond(\partial \varOmega)$ is the Fr\'echet derivative of $\varLambda(\sigma)$ and the {\em divided difference} is extended for recurrent eigenvalues in the standard limit sense,~i.e.,
$$
\frac{\log(\mu_j) - \log(\mu_k)}{\mu_j - \mu_k} = \frac{1}{\mu_j} 
$$
if $\mu_j = \mu_k$.
\end{theorem}

\begin{proof}
Obviously, the finite-dimensional approximation $\sigma \mapsto \varLambda^{(n)}(\sigma)$ inherits Fr\'echet differentiability from $\sigma \mapsto \varLambda(\sigma)$. In fact, the derivative of the former at $\sigma \in L^\infty_+(\varOmega)$ is simply 
\begin{equation}
\label{eq:finder}
L^\infty(\varOmega) \ni \eta  \mapsto P^{(n)} D\varLambda(\sigma;\eta) P^{(n)} \in \mathcal{L}(L^2_\diamond(\partial \varOmega)).
\end{equation}
The principal branch of a logarithm of a matrix (or of a finite-dimensional linear operator given in a fixed basis) is also Fr\'echet differentiable over the set of matrices that do not have eigenvalues on the closed negative real axis in the complex plain~\cite{Higham08}. The Fr\'echet differentiability of the mapping \eqref{eq:logfmap} thus follows from the chain rule for Banach spaces, which also gives the representation \eqref{eq:derrepre} when combined with \cite[Theorem~3.11 \& Corollary~3.12]{Higham08} and \eqref{eq:finder}.
\end{proof}

If one chooses the discretization frame to be $\psi_k = \phi_k \in H^{1/2}_\diamond(\partial \varOmega)$, $k=1, \dots, n$,~i.e.,~the first $n$ orthonormal eigenfunctions of $\varLambda(\sigma)$ for a particular, fixed $\sigma \in L^\infty_+(\varOmega)$, then 
\begin{equation}
  \label{eq:logspectral3}
  \log\!\varLambda^{(n)}(\sigma) = P^{(n)} \log\!\varLambda(\sigma) P^{(n)} \colon f \mapsto \sum_{k=1}^n \log(\lambda_k) (f,\phi_k) \, \phi_k,  
\end{equation}
that is, $\log\!\varLambda^{(n)}(\sigma)$ inherits its eigenvalues and eigenfunctions from $\log\!\varLambda(\sigma)$. In this case,
\begin{equation}
\label{eq:derrepre2}
D \log\!\varLambda^{(n)}(\sigma; \eta) \colon f \mapsto \sum_{j=1}^n \sum_{k=1}^n \frac{\log(\lambda_j) - \log(\lambda_k)}{\lambda_j - \lambda_k} (f, \phi_k) \big( D\varLambda(\sigma; \eta)\phi_k, \phi_j \big) \, \phi_j.
\end{equation}
Recall that here the basis for the discretization is fixed. In other words, the projection $P^{(n)}$ in \eqref{eq:logspectral3} is considered invariable when the Fr\'echet derivative is computed, that is, it corresponds all the time to the first $n$ orthonormal eigenfunctions of the unperturbed $\varLambda(\sigma)$, not to those of $\varLambda(\sigma+\eta)$.

\begin{remark}
\label{remark:Fder}
The derivative of the finite-dimensional logarithmic forward operator of EIT,
$$
L^{(n)} := \log\!\varLambda^{(n)} \circ \exp \colon \kappa \mapsto \log\!\varLambda^{(n)}(\mathrm{e}^\kappa), \quad L^\infty(\varOmega) \to \mathcal{L}(L^2_\diamond(\partial \varOmega))
$$
is obtained by simply replacing $D\varLambda(\sigma; \eta)$ with $D \varLambda_{\exp}(\kappa; \eta)$ of \eqref{eq:kappaderiv} in \eqref{eq:derrepre}.
\end{remark}

\section{Complete electrode model}
\label{sec:cem}

Practical EIT measurement setups are characterized by $M \in \N \setminus \{1\}$ electrodes that are attached to the surface of the imaged object.
Let $E_m \subset \partial \varOmega$, $m=1,\ldots,M$, denote these nonempty, open, connected, well-separated sets.
In the CEM, the Neumann boundary condition in \eqref{eq:eitcont} is replaced by~\cite{Cheng89}
\begin{align}
\label{eq:cem}
\sigma \frac{\partial u}{\partial \nu} = 0 \qquad & \text{on } \partial \varOmega \setminus \bigcup_{m=1}^M E_m, \\[-1mm]
\sigma \frac{\partial u}{\partial \nu} = \zeta_m (U_m-u) \qquad & \text{on } E_m, \quad m=1,\ldots,M, \\[1mm]
\int_{E_m} \sigma \frac{\partial u}{\partial \nu} \, {\rm d} S = I_m, \qquad & m=1,\ldots,M,
\end{align}
where $\zeta = [\zeta_m]_{m=1}^M \in \R_+^M$ denote the (constant) contact conductances at the electrode-object interfaces.
The current pattern $I = [I_m]_{m=1}^M$ is prescribed, whereas the electrode potentials $U = [U_m]_{m=1}^M$ are part of the solution. Both of these vectors belong to the mean-free subspace $\C_\diamond^M \subset \C^M$, the current pattern due to the conservation of electric charge and the electrode potentials as a result of our specific choice for the ground level of potential.  
The variational formulation corresponding to the conductivity equation from \eqref{eq:eitcont} with the boundary conditions \eqref{eq:cem} is to find $(u,U) \in \mathcal{H}^1(\varOmega) := H^1(\varOmega) \oplus \C^M_\diamond$ such that
\begin{equation}
\label{eq:cemvar}
\int_\varOmega \sigma \nabla u \cdot \nabla \overline{v} \, {\rm d} x + \sum_{m=1}^M \zeta_m \int_{E_m} (U_m-u)(\overline{V}_m-\overline{v}) \, {\rm d} S = I \cdot \overline{V}
\end{equation}
holds for all $(v,V) \in \mathcal{H}^1(\varOmega)$. The problem \eqref{eq:cemvar} is uniquely solvable~\cite{Somersalo92}.

For simplicity, we use the same notations for the forward operators of CEM as introduced previously for the continuum model in Section~\ref{sec:param}. However, as a disparity, the CEM forward operators depend on the contact conductances in addition to the conductivity. As an example, the CEM forward operator that takes the conductivity and contact conductances as its inputs is of the form
\begin{equation}
\varLambda \colon (\sigma,\zeta) \mapsto \varLambda(\sigma,\zeta), \quad L_+^\infty(\varOmega) \times \R_+^M \to \C^{(M-1) \times (M-1)},
\end{equation}
where the output is called the resistance matrix (regardless of the used input parametrization), which is given with respect to a fixed orthonormal basis for $\C_\diamond^M$. In other words, if $\alpha \in \C^{M-1}$ carries the coordinates of an electrode current pattern $I \in \C_\diamond^M$ with respect to the employed basis, then $\varLambda(\sigma, \zeta) \alpha$ gives the coordinates of the resulting electrode potentials $U\in \C_\diamond^M$ in that same basis. Since $\sigma$ and $\zeta$ are real-valued and positive, it follows easily from \eqref{eq:cemvar} that $\varLambda(\sigma, \zeta)$ is Hermitian and, in particular, positive-definite. Similarly, 
\begin{equation}
\label{eq:CEMfw}
\varLambda_{\rm inv} \colon (\rho, z) \mapsto \varLambda(\rho^{-1}, z^{-1}), \quad L_+^\infty(\varOmega) \times \R_+^M \to \C^{(M-1) \times (M-1)},
\end{equation}
and
\begin{equation}
\varLambda_{\rm exp} \colon (\kappa, \upsilon) \mapsto \varLambda({\rm e}^\kappa, {\rm e}^\upsilon), \quad L^\infty(\varOmega) \times \R^M \to \C^{(M-1) \times (M-1)},
\end{equation}
where we have abused the notation by defining $z^{-1} = [z_m^{-1}]_{m=1}^M$ and ${\rm e}^\upsilon = [{\rm e}^{\upsilon_m}]_{m=1}^M$ as vectors of $\R^M$. Finally, the logarithmic forward operator is defined in accordance with Definition~\ref{def:logarithmic}, that is, 
\begin{equation}
L \colon (\kappa, \upsilon) \mapsto \log\!\varLambda_{\exp}(\kappa, \upsilon), \quad L^\infty(\varOmega) \times \R^M \to \C^{(M-1) \times (M-1)},
\end{equation}
where $\log\!\varLambda_{\exp}(\kappa, \upsilon)$ is the principal logarithm of the positive-definite matrix $\varLambda_{\exp}(\kappa, \upsilon)$.

The Fr\'echet derivative of the basic forward operator \eqref{eq:CEMfw} with respect to $\sigma \in L^\infty_+(\varOmega)$ in the direction $\eta \in L^\infty(\varOmega)$ is characterized by~\cite{Kaipio00}
\begin{equation}
\label{eq:CEMder1}
\beta^{\rm T} D_\sigma\varLambda(\sigma, \zeta; \eta) \alpha = -\int_\varOmega \eta \nabla u_{\beta} \cdot \nabla u_{\alpha} \, {\rm d} x, \qquad \alpha, \beta \in \C^{M-1},
\end{equation}
where $(u_{\beta}, U_{\beta})$ and $(u_{\alpha}, U_\alpha)$ are the solutions of \eqref{eq:cemvar} corresponding to the current patterns with the coordinates $\alpha$ and $\beta$, respectively, in the employed orthonormal basis for $\C_\diamond^{M}$. Analogously, the derivative of $\zeta \mapsto \varLambda(\sigma, \zeta)$ in the direction $\xi \in \R^M$ can be assembled via
\begin{equation}
\label{eq:CEMder2}
\beta^{\rm T} D_\zeta \varLambda(\sigma, \zeta; \xi) \alpha = -\sum_{m=1}^M \xi_m \int_{E_m} \big( (U_{\beta})_m - u_{\beta} \big) \big( (U_{\alpha})_m - u_{\alpha} \big) \, {\rm d} S, \qquad \alpha, \beta \in \C^{M-1}.
\end{equation}
Furthermore, the Fr\'echet derivatives of $\varLambda_{\rm inv}$ and $\varLambda_{\exp}$ can be easily obtained by utilizing the chain rule and the above two formulas for $\varLambda$  (cf.~\eqref{eq:rhoderiv} and \eqref{eq:kappaderiv}), and those of the logarithmic forward operator $L$ can be deduced by writing an eigendecomposition for $\varLambda(\sigma,\zeta)$ and mimicking Theorem~\ref{thm:Fder} and Remark~\ref{remark:Fder}.

\begin{remark}
Applying the three basic operations ${\rm id}$, ${\rm inv}$ and $\exp$ to the conductivity and the contact conductances, one could come up with nine different parametrizations for the forward operator of the CEM. Here we only consider three of them, although it is probable that in some cases one of the other options is optimal from the standpoint of the linearization error.
\end{remark}

\begin{remark}
\label{remark:high_resistance}
When the contact conductances are perfect, i.e., in the limit $\zeta_m \to \infty$ for all $m = 1, \dots, M$, the resulting model is called the shunt model~\cite{Cheng89,Darde16}.
On the other hand, when $\sigma/\zeta_m$ approaches infinity for all $m = 1, \dots, M$, the model formally approaches a resistor network for which the resistance between electrodes $E_l$ and $E_m$ is $R_{lm} = (\zeta_l |E_l|)^{-1} + (\zeta_m |E_m|)^{-1}$. In particular, if $\zeta_l^{-1} = \zeta_m^{-1} =: \tilde{z}$ and $|E_m| = |E_l| =: A$, then $R_{lm} = 2 \tilde{z} A^{-1}$. Hence, for high contact resistances $z$ and low resistivity $\rho$, it is reasonable to expect that of the different forward operators $\varLambda_{\rm inv}$ is closest to being linear.
\end{remark}

\section{Numerical experiments}
\label{sec:numerics}

In this section, we study the accuracies of different linearization techniques by performing numerical experiments in the unit disk $\varOmega \subset \R^2$.
The elliptic forward problems for both the continuum model and the CEM are solved with the finite element method (FEM) on meshes with approximately $30\,000$ nodes and piecewise linear basis functions. These meshes and corresponding FEM solutions are also employed when evaluating derivatives of forward maps (cf.,~e.g.,~\eqref{eq:sigmaderiv}--\eqref{eq:kappaderiv}).
The discretization errors of the FEM solutions can be regarded to be small.
Throughout this section, the notation $\E[\cdot]$ is used to denote the sample mean.
The size of the sample is $50\,000$ in each considered case and the random realizations of involved parameters are mutually independent.

\subsection{Forward accuracy}
\label{ssec:forward}

We study the linearization errors of the EIT forward problem with discrete lognormal random conductivity fields.
To this end, the domain $\varOmega$ is divided into $N=1800$ subdomains $\varOmega_i$, each having approximately the same area, and the piecewise constant conductivity is written as
\begin{equation}
\label{eq:discrcond}
\sigma(x) = \sum_{i=1}^N \hat{\sigma}_i \chi_i(x),
\end{equation}
where $\chi_i$ is the indicator function of $\varOmega_i$ and the vector $\hat{\sigma} \in \R_+^N$ contains the corresponding conductivity values.
The resistivity $\rho = \sigma^{-1}$ and the logarithm of the conductivity $\kappa = \log(\sigma)$ are defined similarly via vectors $\hat{\rho} \in \R_+^N$ and $\hat{\kappa} \in \R^N$, respectively.
Furthermore, let $\hat{x}_i \in \varOmega_i$ denote the center of a subdomain.
The random conductivities are drawn via their logarithms such that the vectors $\hat{\kappa}$ follow a Gaussian distribution with the probability density
\begin{equation}
p(\hat{\kappa}) = \frac{1}{\sqrt{(2 \pi)^N \lvert \varGamma \rvert}} \exp\mathopen{}\left( -\frac{1}{2}(\hat{\kappa} - \hat{\kappa}_0)^\top \varGamma^{-1} (\hat{\kappa} - \hat{\kappa}_0) \right)\mathclose{},
\end{equation}
where $\hat{\kappa}_0 \in \R^N$ specifies the discrete mean field and
\begin{equation}
\label{eq:covmat}
\varGamma_{i,j} = \varsigma^2 \exp\mathopen{}\left( -\frac{\lVert \hat{x}_i-\hat{x}_j \rVert_2^2}{2\ell^2} \right)\mathclose{}, \qquad i,j=1,\ldots,N
\end{equation}
defines the covariance matrix for some variance parameter $\varsigma^2>0$ and correlation length $\ell>0$.
Table~\ref{tab:randfields} lists the used values for four different distributions.

First we consider the continuum model.
For the purpose of numerical computations, the finite-dimensional approximation $\varLambda^{(n)}(\sigma)$ in \eqref{eq:lambdan} is considered with $n=16$ trigonometric basis functions 
\begin{equation}
\label{eq:discbasis}
\psi_{2k-1}(\theta) = \frac{1}{\sqrt{\pi}} \cos(k\theta), \quad \psi_{2k}(\theta) = \frac{1}{\sqrt{\pi}} \sin(k\theta), \qquad k=1,\ldots,8,
\end{equation}
where $\theta$ denotes the polar angle. 
In order to obtain a fully discrete forward operator, the conductivity $\sigma$ is replaced by $\hat{\sigma}$, resulting in
\begin{equation}
\label{eq:fulldiscr}
\hat{\varLambda} \colon \hat{\sigma} \mapsto \hat{\varLambda}(\hat{\sigma}), \quad \R_+^N \to \R^{n \times n},
\end{equation}
where $\hat{\varLambda}(\hat{\sigma})$ is the matrix representation of $P^{(n)} \varLambda(\sigma) P^{(n)}$ for the conductivity \eqref{eq:discrcond}; see~\eqref{eq:projspan}.
This nonlinear matrix-valued function is linearized around the unit conductivity $\sigma_0 \equiv 1$ as
\begin{equation}
\label{eq:lambdalin}
\hat{\varLambda}(\hat{\sigma}) \approx \hat{\varLambda}_\text{lin}(\hat{\sigma}) \coloneqq \hat{\varLambda}(\hat{\sigma}_0) + \hat{\varLambda}'(\hat{\sigma}_0)(\hat{\sigma}-\hat{\sigma}_0),
\end{equation}
where $\hat{\varLambda}'(\hat{\sigma}) \hat{\eta}$ is the Fr\'echet derivative of the mapping \eqref{eq:fulldiscr} in the direction $\hat{\eta} \in \R^N$. To be more precise,
\begin{equation}
\label{eq:discrderiv}
\hat{\varLambda}'(\hat{\sigma}) \hat{\eta} \, = \, \sum_{i=1}^N \hat{\eta}_i P^{(n)} D\varLambda(\sigma;\chi_i) P^{(n)},
\end{equation}
which can be evaluated with the help of \eqref{eq:sigmaderiv}; see also \eqref{eq:finder}.
The alternative fully discrete forward operators $\hat{\rho} \mapsto \hat{\varLambda}_{\rm inv}(\hat{\rho})$ and $\hat{\kappa} \mapsto \hat{\varLambda}_{\exp}(\hat{\kappa})$ are defined analogously, as are their linearizations around $\rho_0 \equiv 1$ and $\kappa_0 \equiv 0$, respectively, with the appropriate derivative from \eqref{eq:rhoderiv} or \eqref{eq:kappaderiv} replacing $D\varLambda(\sigma;\chi_i)$ in \eqref{eq:discrderiv}. The same conclusions also apply to the fully discrete logarithmic forward operator $\hat{\kappa} \mapsto \hat{L}(\hat{\kappa})$; cf.~Definition \ref{def:logarithmic}, Theorem~\ref{thm:Fder} and Remark~\ref{remark:Fder} with $\eta = \chi_i$.
Note that for a rotationally symmetric conductivity, \eqref{eq:discbasis} are eigenfunctions of the Neumann-to-Dirichlet operator, meaning that $\varphi_k = \phi_k = \psi_k$ in the notation of Section~\ref{sec:log} (cf.~\cite{Isaacson1986}). In particular, \eqref{eq:logspectral3} and \eqref{eq:derrepre2} hold.

\begin{table}
  \caption{Parameters defining the discrete random conductivity fields F1--F4 and the random contact conductances C1--C2 used in the numerical experiments.}
  \begin{center}
  \begin{tabular}{l | c | c | c | c | c}
  & $N$ & $\hat{\kappa}_0$ & $\varsigma^2$ & $\ell$ & $\E[\lVert \kappa \rVert_{L^2(\varOmega)}]$ \bigstrut[b] \\
  \hline
  F1 & $1800$ & $0$ & $1/4$ & $1/3$ & $1.53$ \bigstrut[t] \\
  F2 & $1800$ & $0$ & $1/4$ & $2/3$ & $1.48$ \\
  F3 & $1800$ & $0$ & $1$ & $1/3$ & $3.08$ \\
  F4 & $1800$ & $0$ & $1$ & $2/3$ & $2.95$ \\
  \end{tabular}
  \quad
  \begin{tabular}{l | c | c | c | c }
  & $M$ & $\upsilon_0$ & $\gamma^2$ \bigstrut[b] \\
  \hline
  C1 & $16$ & $\log(10)$ & $1$ \bigstrut[t] \\
  C2 & $16$ & $\log(1000)$ & $1$ \\ \\ \\
  \end{tabular}
  \end{center}
  \label{tab:randfields}
\end{table}

The error indicator for the linearized $\hat{\varLambda}$ is  
\begin{equation}
e(\hat{\varLambda}) \coloneqq \E\mathopen{}\left[ \frac{ \lVert \hat{\varLambda}_\text{lin}(\hat{\sigma}) - \hat{\varLambda}(\hat{\sigma}) \rVert_F}{\lVert \hat{\varLambda}(\hat{\sigma})\rVert_F} \right]\mathclose{},
\end{equation}
where $\lVert \cdot \rVert_F$ is the Frobenius norm.
The corresponding definitions of indicators for $\hat{\varLambda}_{\rm inv}$ and $\hat{\varLambda}_{\exp}$ should be obvious. 
For the logarithmic forward operator, the error is computed as
\begin{equation}
e(\hat{L}) \coloneqq \E\mathopen{}\left[ \frac{ \lVert \exp(\hat{L}_\text{lin}(\hat{\kappa})) - \hat{\varLambda}(\hat{\sigma}) \rVert_F}{\lVert \hat{\varLambda}(\hat{\sigma})\rVert_F} \right]\mathclose{},
\end{equation}
where $\exp$ is just the ordinary matrix exponent,~i.e.,~the inverse of (the principal branch of) the matrix logarithm.

The left-hand side of Table \ref{tab:forwerr} lists the mean errors for the four linearization approaches and for the four different random fields defined in Table~\ref{tab:randfields}.
It can be seen that linearizing with respect to the conductivity results in the worst accuracy in each case, whereas the novel logarithmic linearization method is clearly the best.
Increasing the correlation length in the random field deteriorates the performance of the conductivity and log-conductivity linearizations, whereas $\hat{\varLambda}_{\rm inv}$ and $\hat{L}$ perform better with the smoother fields. An intuitive explanation for the latter phenomenon is that increasing the correlation length in our random model makes the corresponding realizations of the conductivity to be closer to constants, for which $\hat{\varLambda}_{\rm inv}$ and $\hat{L}$ are linear and affine, respectively; see Examples \ref{ex:const} and \ref{ex:const2}.  It is not surprising that increasing the pointwise variance $\varsigma^2$ leads in every case to a higher sample mean for the linearization error.

\begin{table}
  \caption{Linearization errors $e$ for the continuum forward model (left) and for the CEM (right) with the different random distributions defined in Table~\ref{tab:randfields}. The forward operators parametrized with respect to conductivity, resistivity and log-conductivity are denoted by $\hat{\varLambda}$, $\hat{\varLambda}_{\rm inv}$ and $\hat{\varLambda}_{\exp}$, respectively, whereas the logarithmic forward operator is $\hat{L}$.}
  \begin{center}
  \begin{tabular}{l | c | c | c | c }
  & $e(\hat{\varLambda})$ & $e(\hat{\varLambda}_{\rm inv})$ & $e(\hat{\varLambda}_{\exp})$ & $e(\hat{L})$ \bigstrut[b] \\
  \hline
  F1 & $0.237$ & $0.076$ & $0.095$ & $0.039$ \bigstrut[t] \\
  F2 & $0.294$ & $0.043$ & $0.123$ & $0.027$ \\
  F3 & $0.989$ & $0.288$ & $0.327$ & $0.142$ \\
  F4 & $1.481$ & $0.146$ & $0.434$ & $0.094$ \\
  \end{tabular}
  \quad
  \begin{tabular}{l | c | c | c | c }
  & $e(\hat{\varLambda})$ & $e(\hat{\varLambda}_{\rm inv})$ & $e(\hat{\varLambda}_{\exp})$ & $e(\hat{L})$ \bigstrut[b] \\
  \hline
  F3/C1 & $0.916$ & $0.175$ & $0.390$ & $0.270$ \bigstrut[t] \\
  F4/C1 & $1.012$ & $0.090$ & $0.418$ & $0.269$ \\
  F3/C2 & $1.003$ & $0.272$ & $0.345$ & $0.132$ \\
  F4/C2 & $1.420$ & $0.139$ & $0.430$ & $0.090$ \\
  \end{tabular}
  \end{center}
  \label{tab:forwerr}
\end{table}

For the CEM experiments we employ the same lognormal conductivity fields F3--F4 as with the continuum model, but now the contact conductances $\zeta_m$ are also random. The conductances are drawn independently and for each of the $M=16$ electrodes the conductance distribution is lognormal with the underlying Gaussian distribution having mean $\upsilon_0$ and variance $\gamma^2$ as listed on the right in Table~\ref{tab:randfields}. The identical electrodes are positioned uniformly and they cover approximately $46 \%$ of the boundary $\partial \varOmega$. The orthonormal electrode current basis for $\R_\diamond^M$, employed when forming the resistance matrices, is the discrete counterpart of the Fourier basis \eqref{eq:discbasis}, that is,
$$
I^{(2k - 1)} = \sqrt{\frac{2}{M}} \, [\cos(k \theta_m)]_{m=1}^{M}, \quad I^{(2k)} = \sqrt{\frac{2}{M}}\, [\sin(k \theta_m)]_{m=1}^{M}, \qquad k=1,\ldots,M/2-1,
$$
and $I^{(M-1)} = [(-1)^{m-1}]_{m=1}^M/\sqrt{M}$. Here, $\theta_m = 2\pi(m-1)/M$ is the central polar angle of the $m$th electrode.  The definitions of the fully discrete CEM forward operators, their linearizations and the corresponding error indicators should be obvious.
In particular, the forward operators are now mappings from $\R_+^{N+M}$ or $\R^{N+M}$ to $\R^{(M-1) \times (M-1)}$ and their derivatives can be assembled using the appropriate variants of \eqref{eq:CEMder1}, \eqref{eq:CEMder2} and \eqref{eq:derrepre2}.

For the contact conductance distribution C1 with the lower expected value, the resistivity-based operator $\varLambda_{{\rm inv}}$ results in the best linearization accuracy, as can be seen on the right in Table~\ref{tab:forwerr}. This is in line with Remark~\ref{remark:high_resistance}. On the other hand, when the contact conductance is high (i.e., the setting is close to the shunt model), the errors almost coincide with the corresponding values for the continuum model.
That is, the rows corresponding to C2 are similar to the adjacent rows on the left-hand side of Table~\ref{tab:forwerr}. In particular, the logarithmic forward operator $\hat{L}$ gives the best and the conductivity-based parametrization the worst linearization accuracy for the CEM as well.

\subsection{Inverse accuracy}
\label{ssec:inverse}

Regarding the accuracy of linearizations when solving the inverse problem of EIT, we only consider the forward operators $\hat{\varLambda}_{\exp}$ and $\hat{L}$ that are based on the logarithmic input parametrization.
One reason for this is the fact that non-logarithmic parametrizations may lead to conductivity reconstructions that are not positive.
In addition, it would be problematic to design equivalent regularizations for different input parametrizations, and thus the error indicators would not be directly comparable to each other.
In the following, we give all definitions for the logarithmic forward operator $\hat{L}$, but the corresponding definitions for $\hat{\varLambda}_{\exp}$ should be obvious.

The simulated finite-dimensional Neumann-to-Dirichlet operators are denoted by $\widetilde{\varLambda}(\hat{\sigma})$.
They are also computed with FEM, but in order to avoid inverse crimes, the matrices $\widetilde{\varLambda}(\hat{\sigma})$ are contaminated with Gaussian noise.
More precisely, realizations of independent, zero-mean normal random variables with standard deviation
\begin{equation}
\delta \coloneqq 10^{-3} \E \mathopen{}\left[ \max_{i,j} \hat{\varLambda}(\hat{\sigma})_{i,j} \right]\mathclose{}
\end{equation}
are added to each element of a matrix $\hat{\varLambda}(\hat{\sigma})$, and the matrix is subsequently symmetricized.
The simulated CEM measurements are obtained in the same way and they are denoted by $\widetilde{\varLambda}(\hat{\sigma},\zeta)$.

For a regularization parameter $t>0$, the reconstruction $\hat{\kappa}_t$ based on the logarithmic continuum model forward operator $\hat{L}$ is 
\begin{equation}
\label{eq:map}
\hat{\kappa}_t(\hat{L}) = \argmin_{\hat{y} \in \R^N} \mathopen{}\left\{ \big\lVert \hat{L}_\text{lin}(\hat{y}) - \log(\widetilde{\varLambda}(\hat{\sigma}))\big\rVert_F^2 + t^2 \big\lVert G (\hat{y}-\hat{\kappa}_0) \big\rVert_2^2 \right\}\mathclose{},
\end{equation}
where $G \in \R^{N \times N}$ is a `Bayesian regularization matrix' satisfying $\varGamma^{-1} = G^\top G$.
Here, $\varGamma$ is the covariance matrix \eqref{eq:covmat} of the distribution from which the logarithm of the vector $\hat{\sigma}$ is drawn, and the expected log-conductivity $\hat{\kappa}_0$ is as in Section~\ref{ssec:forward}. The reconstruction obtained from \eqref{eq:map} approximates a \emph{maximum a posteriori} (MAP) estimate under the assumption that the linearization error is negligible and the elements of the matrix $\log(\widetilde{\varLambda}(\hat{\sigma}))$ are contaminated by independent realizations of zero-mean Gaussian noise with standard deviation $t>0$ \cite{Kaipio04a}. (Notice that the latter is not true for any $t$ due to the application of the matrix logarithm after adding the measurement noise.)
Naturally, for $\hat{\varLambda}_{\exp}$ the minimization does not involve a logarithm, but the linearized operator is compared directly to the matrix $\widetilde{\varLambda}(\hat{\sigma})$ in the Frobenius norm.
For the CEM forward operators, the reconstructions are computed according to
\begin{equation}
\label{eq:cemmap}
\big(\hat{\kappa}_t(\hat{L}), \upsilon_t(\hat{L})\big) = \argmin_{\hat{y} \in \R^N, \, w \in \R^M} \mathopen{}\left\{ \big\lVert \hat{L}_\text{lin}(\hat{y},w) - \log(\widetilde{\varLambda}(\hat{\sigma},\zeta))\big\rVert_F^2 + t^2 \big\lVert G (\hat{y}-\hat{\kappa}_0) \big \rVert_2^2 \right\}\mathclose{}
\end{equation}
or by replacing $\hat{L}_\text{lin}(\hat{y},w)$ with $(\hat{\varLambda}_{\exp})_{\rm lin}(\hat{y}, w)$ and deleting the logarithm. In particular, no regularization is applied to the contact conductances since their estimation is a relatively stable task (cf, e.g.,~\cite[(4.8)]{Harhanen15}). Observe that \eqref{eq:map} and \eqref{eq:cemmap} correspond to minimizing quadratic Tikhonov functionals, which is a simple task.

The average reconstruction error for a fixed regularization parameter $t>0$ is
\begin{equation}
\iota_t(\hat{L}) \coloneqq \E \mathopen{}\left[ \big\lVert \kappa_t(\hat{L}) - \kappa \big\rVert_{L^2(\varOmega)} \right]\mathclose{},
\end{equation}
where $\kappa = \log(\sigma)$ is the (random draw of a) true log-conductivity corresponding to the vector $\hat{\sigma}$ in \eqref{eq:map}; see~\eqref{eq:discrcond}.
For a given conductivity/conductance distribution, the optimal regularization parameter is obtained by solving a one-dimensional optimization problem,
\begin{equation}
\tau(\hat{L}) \coloneqq \argmin_{t \in \R_+} \iota_t(\hat{L}),
\end{equation}
and the corresponding error indicator is simply denoted by $\iota(\hat{L}) \coloneqq \iota_{\tau(\hat{L})}(\hat{L})$. (Naturally, such an optimal regularization parameter cannot be found in practice when the true conductivity is not known, but the idea here is to compare upper bounds for the performance of the two considered one-step reconstruction algorithms.)
In addition, for the CEM we denote the contact conductance reconstruction error by
\begin{equation}
d(\hat{L}) \coloneqq \E\mathopen{}\left[ \lVert \upsilon_{\tau(\hat{L})}(\hat{L}) - \upsilon \rVert_2 \right]\mathclose{},
\end{equation}
where $\upsilon = \log(\zeta)$ is defined componentwise, but it should be emphasized that the optimal regularization parameter $\tau(\hat{L})$ is still chosen so that it minimizes the $L^2(\varOmega)$ reconstruction error corresponding to the mere log-conductivity.  Once again, the corresponding definitions for $\varLambda_{\exp}$ are analogous.

\begin{table}
  \caption{Conductivity reconstruction errors $\iota$ for the linearized continuum model (left) and for the linearized CEM (right), together with the conductance errors $d$ for the CEM. The considered random fields and parameters are described in Table~\ref{tab:randfields}.}
  \begin{center}
  \begin{tabular}{l | c | c }
  & $\iota(\hat{\varLambda}_{\exp})$ & $\iota(\hat{L})$ \bigstrut[b] \\
  \hline
  F1 & $0.583$ & $0.362$ \bigstrut[t] \\
  F2 & $0.524$ & $0.103$ \\
  F3 & $2.134$ & $1.028$ \\
  F4 & $2.000$ & $0.283$ \\
  \end{tabular}
\quad 
  \begin{tabular}{l | c | c | c | c}
  & $\iota(\hat{\varLambda}_{\exp})$ & $\iota(\hat{L})$ & $d(\hat{\varLambda}_{\exp})$ & $d(\hat{L})$ \bigstrut[b] \\
  \hline
  F3/C1 & $1.847$ & $1.477$ & $8.266$ & $3.404$ \bigstrut[t] \\
  F4/C1 & $1.866$ & $0.775$ & $8.034$ & $2.889$ \\
  F3/C2 & $1.944$ & $1.346$ & $421.3$ & $164.6$ \\
  F4/C2 & $1.880$ & $0.568$ & $359.4$ & $79.98$ \\
  \end{tabular}
  \end{center}
  \label{tab:inverr}
\end{table}

Table~\ref{tab:inverr} reveals that the new logarithmic method is more accurate in each test case, and the difference is highlighted with the smoother fields F2 and F4, as could have been predicted based on the forward accuracy observations in Table~\ref{tab:forwerr}.
In practice, the choice of regularization parameter may well be suboptimal.
However, Figure~\ref{fig:tau} demonstrates that the new method results in a smaller average error with almost any choices for the regularization parameters.

\begin{figure}
\center{
{\includegraphics[scale=.5]{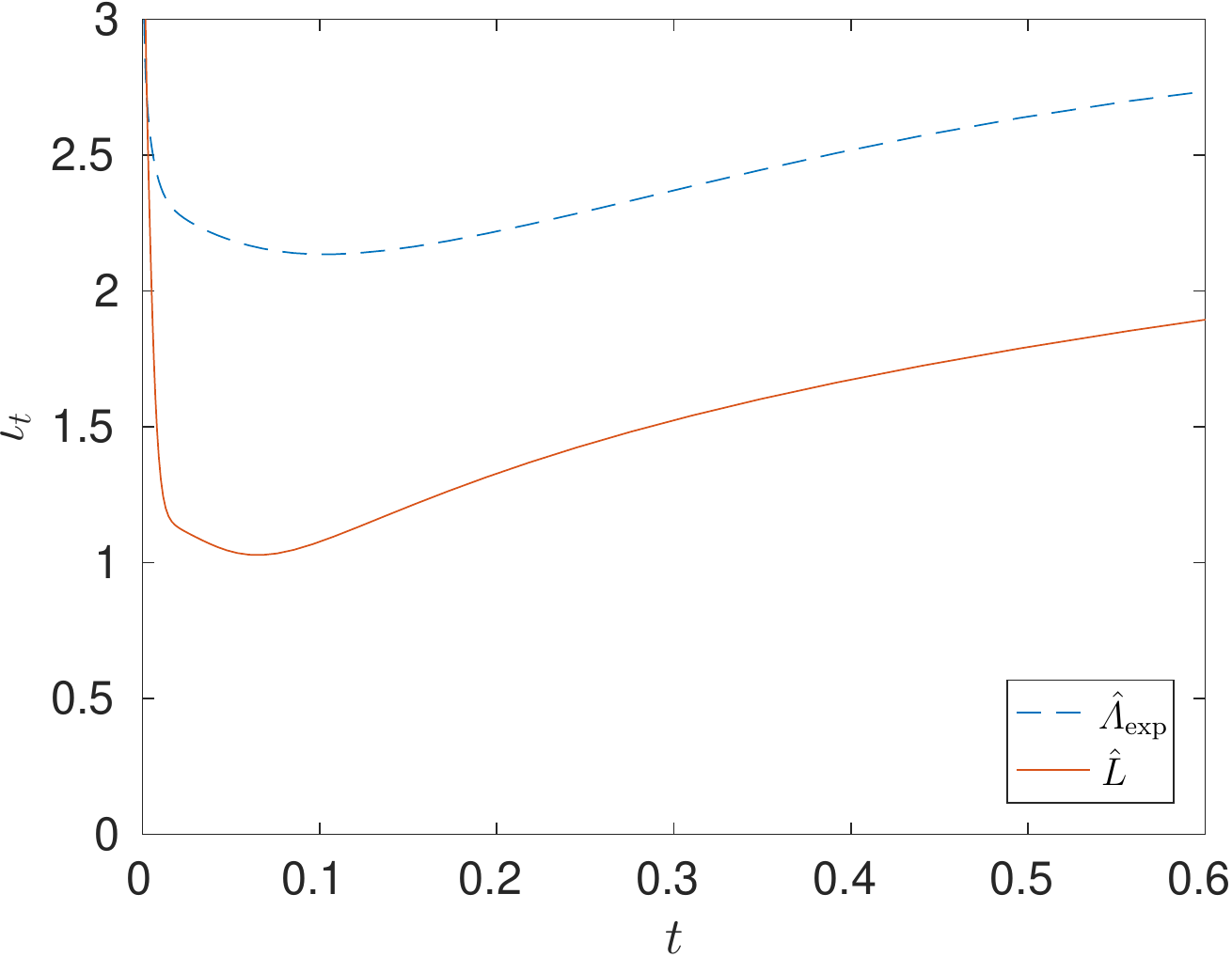}}
\quad
{\includegraphics[scale=.5]{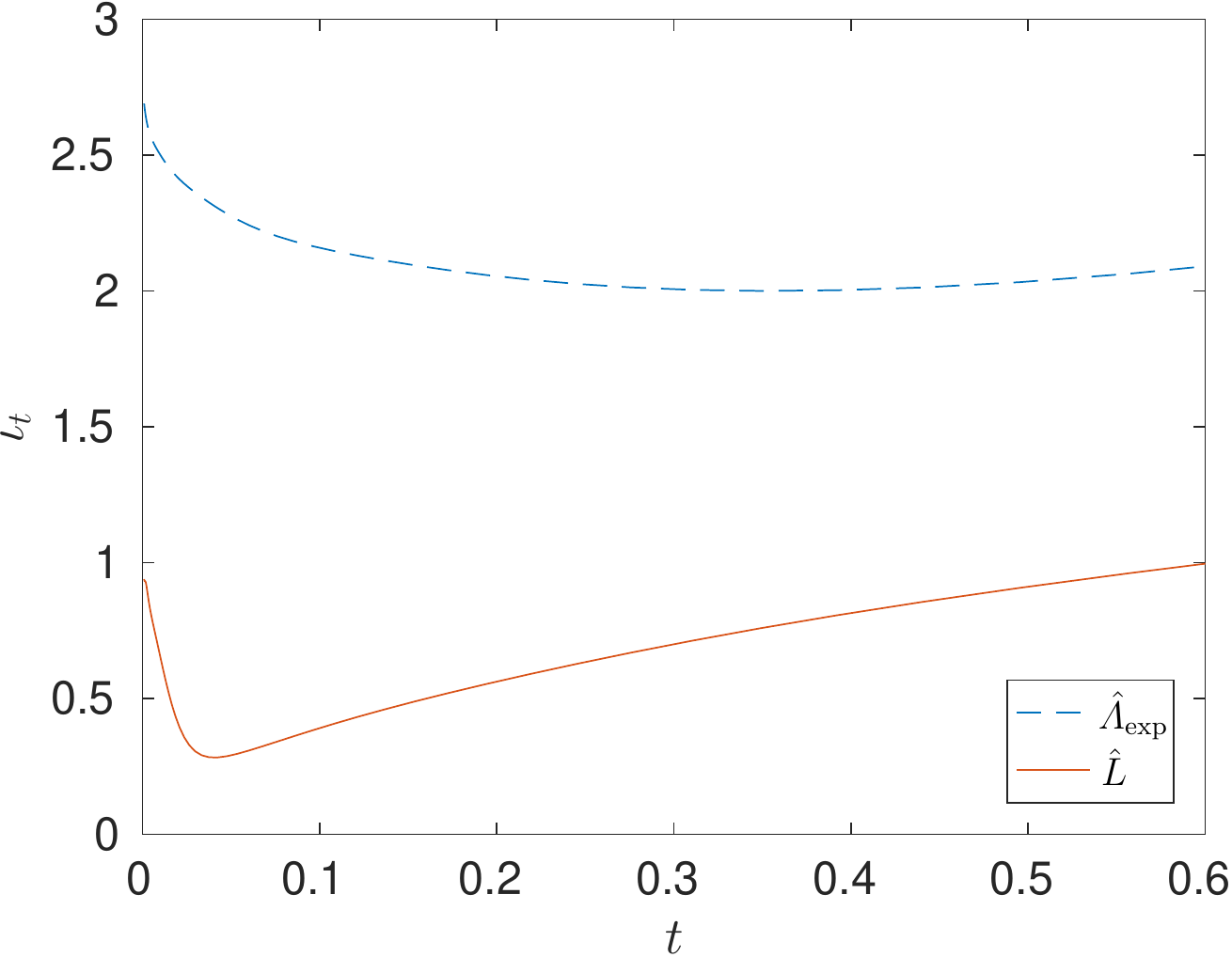}}
}
\caption{Mean reconstruction errors $\iota_t$ as functions of the regularization parameter $t$ for the linearization methods based on the continuum forward operators $\hat{\varLambda}_{\exp}$ (dashed) and $\hat{L}$ (solid). The plots correspond to the random fields F3 (left) and F4 (right).}
\label{fig:tau}
\end{figure}

In order to illustrate the difference in the one-step inversion accuracies for $\hat{\varLambda}_{\exp}$ and $\hat{L}$, some example reconstructions, with regularization parameters that are optimized as described above, are shown in Figures \ref{fig:f3c2recos} and \ref{fig:f4c2recos}.
The first example, which is based on a realization of the random conductivity field F3 and the CEM with high contact conductances, demonstrates that the traditional linearization method is not capable of recovering the high contrast of the target conductivity. On the other hand, the new method reproduces the conductivity levels far more accurately and also recovers the resistive area close to the northwest boundary of the unit disk.
The second example presented in Figure~\ref{fig:f4c2recos} considers a realization of the random conductivity field F4 with twice as long correlation length. The superiority of the new logarithmic method is even more clearly visible with this smoother conductivity field, as was expected based on the statistical evidence in Tables \ref{tab:forwerr} and \ref{tab:inverr}.

\begin{figure}
\center{
{\includegraphics[scale=.45]{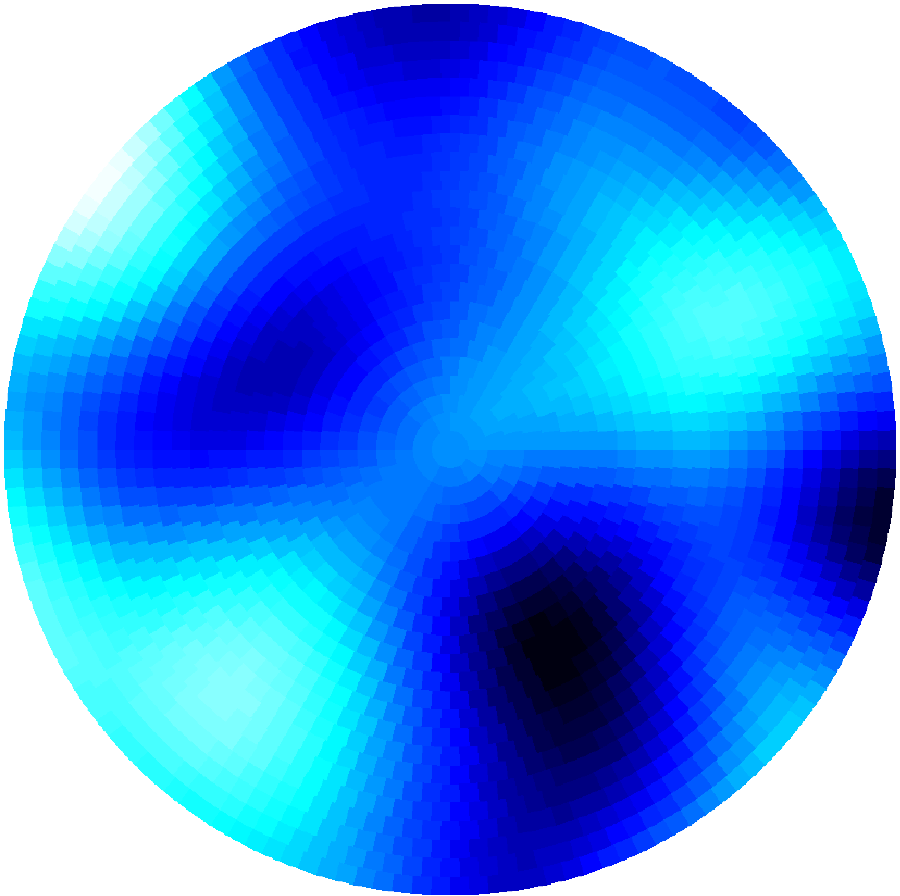}}
\quad
{\includegraphics[scale=.45]{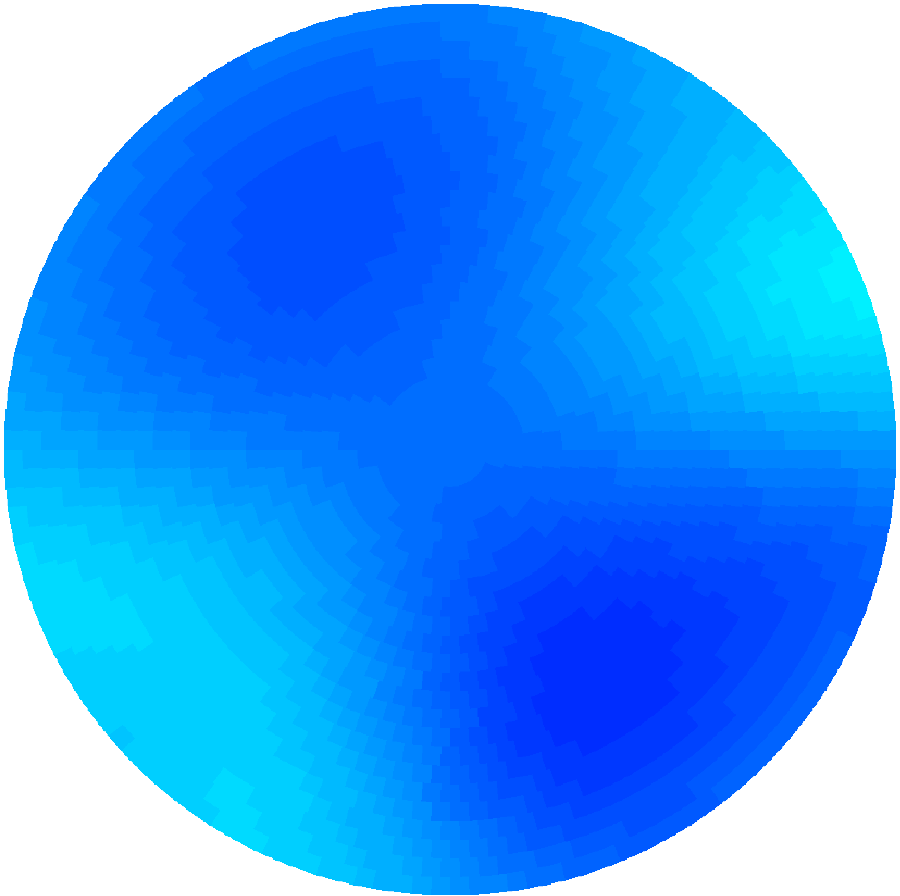}}
\quad
{\includegraphics[scale=.45]{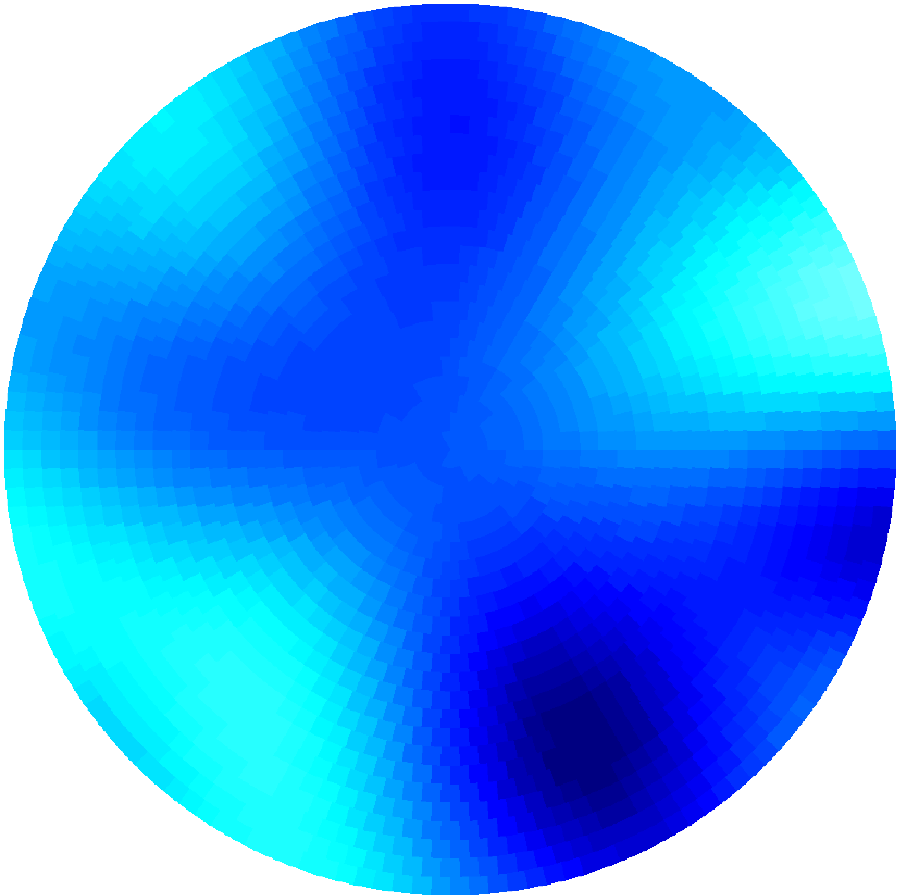}}
\quad
{\includegraphics[scale=.45]{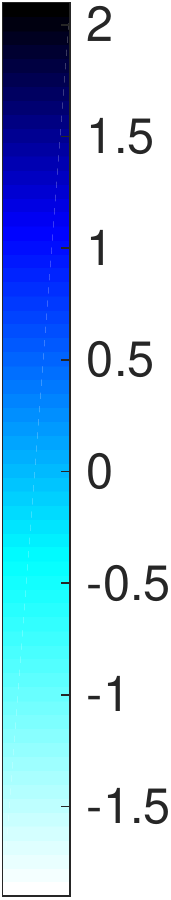}}
}
\caption{Example reconstructions for the CEM and a realization of the random conductivity/conductances F3/C2. Left:\ target log-conductivity with $\lVert \kappa \rVert_{L^2(\varOmega)} \approx 2.705$. Middle:\ reconstruction based on a single linearization of $\hat{\varLambda}_{\exp}$ with $L^2(\varOmega)$-error $1.863$. Right:\ reconstruction based on a single linearization of $\hat{L}$ with $L^2(\varOmega)$-error $1.292$.}
\label{fig:f3c2recos}
\end{figure}

\begin{figure}
\center{
{\includegraphics[scale=.45]{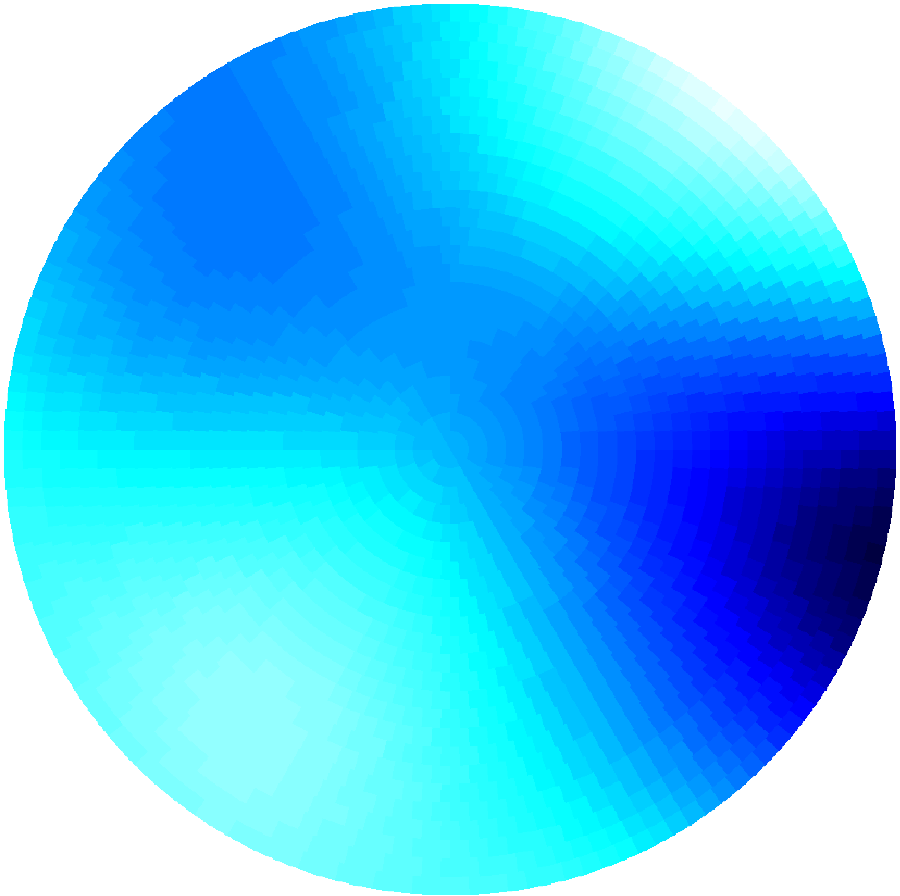}}
\quad
{\includegraphics[scale=.45]{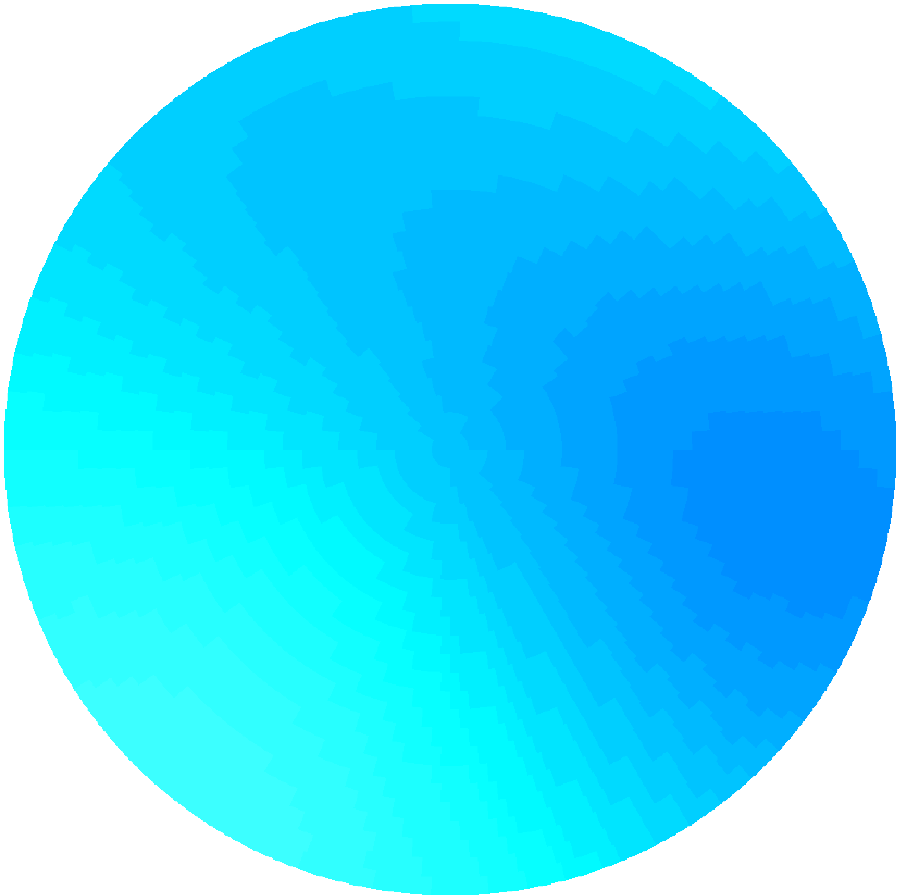}}
\quad
{\includegraphics[scale=.45]{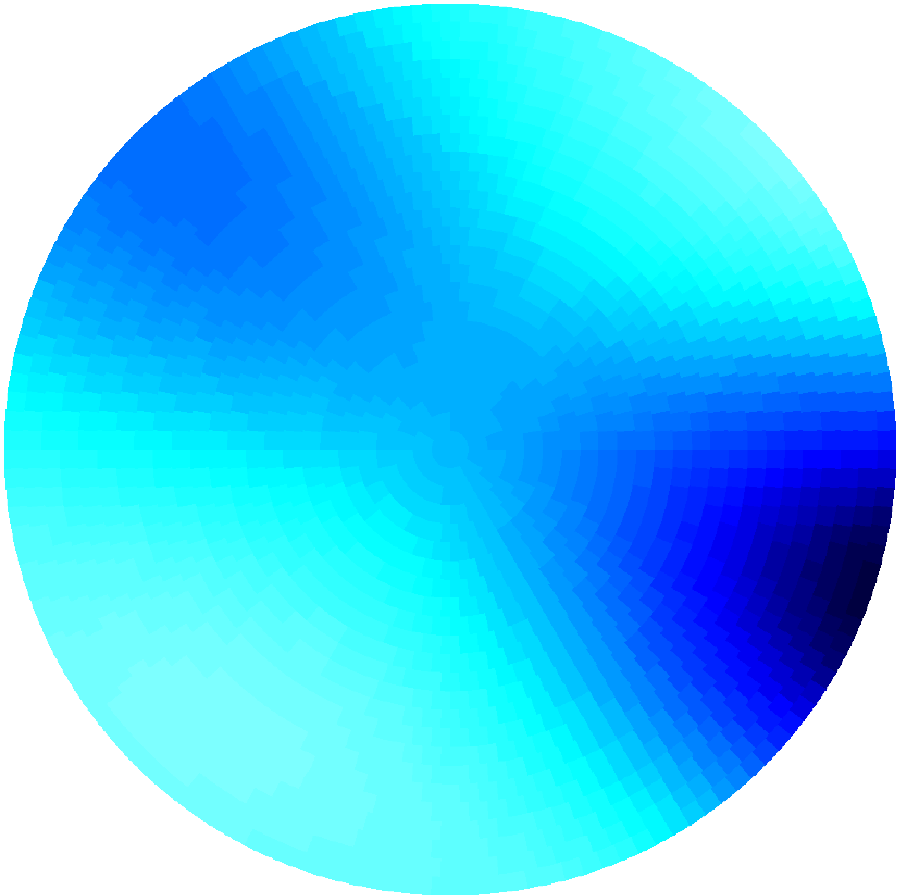}}
\quad
{\includegraphics[scale=.45]{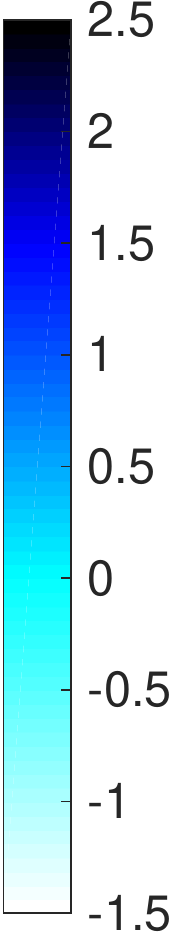}}
}
\caption{Example reconstructions for the CEM and a realization of the random conductivity/conductances F4/C2. Left:\ target log-conductivity with $\lVert \kappa \rVert_{L^2(\varOmega)} \approx 2.431$. Middle:\ reconstruction based on a single linearization of $\hat{\varLambda}_{\exp}$ with $L^2(\varOmega)$-error $1.637$. Right:\  reconstruction based on a single linearization of $\hat{L}$ with $L^2(\varOmega)$-error $0.550$.}
\label{fig:f4c2recos}
\end{figure}

\section{Discussion}
\label{sec:discussion}

We have reviewed existing linearization approaches for EIT and proposed a new logarithmic technique that seems to be the most accurate linearization method amongst those studied here.
Although the EIT forward operator is somewhat frequently linearized with respect to the conductivity, it was numerically demonstrated that the linearization error in the measurement matrix becomes smaller if the linearization is based on the resistivity or the logarithm of the conductivity.
In a sense, this is just a consequence of the Ohm's law, and the conclusion applies to both the continuum model and to the CEM.  Our novel method, which maps the logarithm of the conductivity to the logarithm of the Neumann-to-Dirichlet operator or the measurement matrix, produces the smallest linearization error in almost all tested cases.

Regarding the inverse problem of EIT, the proposed logarithmic method retains the positivity property of the log-conductivity parametrization, while preserving the accuracy of the resistivity linearization in the case of, e.g., constant conductivities. Numerical studies and example reconstructions demonstrate that when comparing to the traditional log-conductivity linearization approach, the proposed method is clearly more accurate, regardless of the parameters defining the lognormal random models for the conductivity and the contact conductances.

The above conclusions are valid if the measurements are modelled by the Neumann-to-Dirichlet map or its counterparts for electrode measurements. If the Dirichlet-to-Neumann map were employed, one would expect conductivity parametrizations to prevail over those based on the resistivity (cf.~Ohm's law). However, such a change would not affect the performance of the novel method since the logarithms of the Neumann-to-Dirichlet and Dirichlet-to-Neumann operators only differ by their sign.

\section*{Acknowledgments}
This work was supported by the Academy of Finland (decision 267789) and the Finnish Foundation for Technology Promotion TES.

\bibliographystyle{lineit}
\bibliography{lineit-refs}

\end{document}